\RequirePackage{fix-cm}
\documentclass[stropt,envcountsect,envcountsame]{svjour3}                     
\smartqed  
\usepackage{graphicx}
\usepackage[noadjust]{cite}
\usepackage{filecontents}
\usepackage[margin=1in,footskip=0.25in]{geometry}
\usepackage[none]{hyphenat}
\usepackage{amsmath,amsfonts,amssymb,amsthm,epsfig,epstopdf,url,array}
\usepackage{textcomp}
\theoremstyle{plain}
\newtheorem{thm}{Theorem}[section]

\newtheorem{cor}{Corollary}

\theoremstyle{definition}
\newtheorem{defn}{Definition}[section]

\theoremstyle{remark}
\newtheorem{rem}{Remark}[section]
\begin{document}

\title{L(t, 1)-Colouring of Graphs
}


\author{Priyanka Pandey         \and
        Joseph Varghese Kureethara 
}


\institute{Priyanka Pandey \at
              Department of Mathematics, Christ University \\
              Tel.: +91-8197826228\\
              \email{priyanka.pandey@maths.christuniversity.in}   \\
           \and
           Joseph Varghese Kureethara \at
             Department of Mathematics, Christ University\\
              Tel.: +91-9341094110\\
             \email{frjoseph@christuniversity.in} \\
}

\date{Received: date / Accepted: date}

\maketitle

\begin{abstract}
	One of the most famous applications of Graph Theory is in the field of Channel Assignment Problems. There are varieties of graph colouring concepts that are used for different requirements of frequency assignments in communication channels. We introduce here L(t, 1)-colouring of graphs. This has its foundation in T-colouring and L(p, q)-colouring. For a given finite set T including zero, an L(t, 1)-colouring of a graph G is an assignment of non-negative integers to the vertices of G such that the difference between the colours of adjacent vertices must not belong to the set T and the colours of vertices that are at distance two must be distinct. The variable t in L(t, 1) denotes the elements of the set T. For a graph G, the L(t, 1)-span of G  is the minimum of the highest colour used to colour the vertices of a graph out of all the possible L(t, 1)-colourings. It is denoted by $\lambda_{t,1} (G)$. We study some properties of L(t, 1)-colouring. We also find upper bounds of $\lambda_{t,1} (G)$ of selected simple connected graphs.
\end{abstract}

\section{Introduction}
\label{intro}
\par Graph colouring problems emerged as a requirement of the partitioning of the set of vertices into various classes for specific needs. The \textit{Channel Assignment Problems} are associated with allocating the radio frequencies to various radio transmitters(channels). The concept of graph colouring  is useful in the optimal assignment of radio frequencies. One of the first types of channel assignment problems was introduced by B. H. Metzger\cite{metz}.

\par A graph G is a triple consisting of a vertex set $V(G)$, an edge set $E(G)$, and a relation that associates with each edge two vertices called its endpoints\cite{west}. When the endpoints of an edge coincides, the edge is called a loop. When there are more than one edge between a pair of vertices the graph is said to have multiple edges. A simple graph is a graph that neither has loops or multiple edges. If no direction is assigned to any of the edges in a graph then the graph is an undirected graph. The endpoints of an edge are also referred to as adjacent vertices. A graph is said to be connected if there exist a sequence of adjacent vertices between any pair of vertices in that graph. For all standard definitions and notations related to graphs we suggest the readers to read \cite{west}.
\par In this work we study only about undirected simple connected graphs. The most important constituents of a graph are its vertices and edges. To distinguish the vertices and edges in a graph with multiple vertices and edges, labels are given to them. Colours are certain types of labels assigned to vertices and edges. If the colours are assigned to the vertices alone, then we call it as vertex colouring of graph. Similarly, we can have edge colouring. When colours are assigned simultaneously to vertices and edges, we call it as total colouring.
\par  Colours are assigned for various practical purposes. Channel assignment problems are a type of colouring problems in graph theory. This is done by representing vertices as the radio transmitters and colours as the radio frequencies. To avoid interference of the radio transmissions, the frequencies of radio transmitters within certain distances must be distinct. In this scenario T-colouring is a useful mathematical tool for optimal frequency allocation.
\subsection{Preliminaries}
\label{sec:2}
 The concept of T-colouring of vertices of a graph was introduced by W. K. Hale\cite{hale}. It was further studied by S. Roberts\cite{fred} in 1991. T-colouring is defined as follows.
\begin{defn}
	For a graph G and a given finite set T of non-negative integers containing 0, a \textit{T-colouring} of G is an assignment \textit{c} of colours (positive integers) to the vertices of G such that if $uw \in E(G)$, then $|c(u)-c(w)| \notin T$ \cite{chrom}.
\end{defn}
Although T-colouring was an answer to frequency assignment requirement, there emerged several other colouring schemes as to meet any new industrial demands. In T-colouring, the condition of colouring was for adjacent transmitters i.e., neighbouring transmitters, whereas transmitters which are neighbours to the neighbours have also the possibility of transmission interference. R. K. Yeh in 1990 and then J. R. Griggs and R. K. Yeh in 1992 addressed this issue\cite{chrom}.
Labelling graphs with a condition at distance 2 was studied extensively by Griggs
and Yeh \cite{phd}. The related colouring problem is known as L(p, q)-colouring which is defined as follows. 
\begin{defn}
	For $p \ge q \ge 1$, an \textit{L(p, q)-colouring} of a graph G=(V, E) is an assignment \textit{c} from V(G) to the set of non-negative integers ${0,1,....,\lambda}$ such that
	\[
	|c(u)-c(w)| \ge 
	\begin{cases}
	p & \text{if}~~ d(u,w)=1\\
	q & \text{if}~~ d(u,w)=2
	\end{cases}
	\]
	where $d(u,w)$ is the distance between the vertices $u$ and $w$\cite{chrom}.
\end{defn}
When $p=2$ and $q=1$, we get L(2, 1)-colouring.
\section{L(t, 1)-colouring}
\label{sec:1}
	\par We introduce a type of colouring taking the ideas from T-colouring and L(2, 1)-colouring.

\begin{defn}
	Let  G=(V, E) be a graph and let $d(u,v$) be the distance between the vertices u and v of G. Let T be a finite set of non-negative integers containing 0. An L(t, 1)-colouring of a graph G is an assignment \textit{c} of non-negative integers to the vertices of G such that $|c(u)-c(v)| \notin T$ if $d(u,v)=1$ and $c(u)\ne c(v)$ if $d(u,v)=2$.
\end{defn} 
	It is not reasonable to assume that there are infinitely many frequencies available for allotment. Therefore, it is of great importance to find the optimal assignment of the available frequencies to radio transmitters. Thus we introduce the concept of span. The span of frequencies have been studied in relation to various types of channel assignment problems\cite{chrom}. We define the L(t, 1)-span of graph G as follows.
\begin{defn}
	For a graph G with a given set T, \textit{L(t, 1)-span of G} denoted by the symbol $\lambda_{t,1} (G)$ is
	\begin{equation}
	\lambda_{t,1} (G) = min~\{\max_{u,v \in V(G)}~\{|c(u)-c(v)| |~c~ \text{is an L(t, 1)-colouring of G}\}\}.
	\end{equation}

\end{defn}
Practically, the highest colour assigned becomes the span of the graph.
	\begin{defn}
		For any colouring $c$, the $c$-span of the colouring is defined as the highest colour assigned to any vertex in graph G.
	\end{defn}
\subsection{Complementary Colouring}
If $c$ is an L(t, 1)-colouring then its induces other L(t, 1)-colouring as well. Complementary colouring is one among them.Let s be the largest colour assigned in an L(t, 1)-colouring \textit{c}. Then for $j\ge0$ the colouring $c'$ of the vertices G defined by 
\begin{equation}
	c'(v)= s + j - c(v) 
\end{equation}
for each vertex $v$ in $V(G)$ is also an L(t, 1)-colouring of G. The colouring $c'$ is called the complementary L(t, 1)-colouring of $c$.\\
\begin{thm}
	If $c$ is an L(t, 1)-colouring of a graph G, then the complementary colouring $c'$ is also an L(t, 1)-colouring of the same graph G.
\end{thm}
\begin{proof}
	Let \textit{c} be an L(t, 1)-colouring of a graph G with s as the largest colour assigned. Then for any two adjacent vertices \textit{u} and \textit{v}, \\We know that $|c(u)-c(v)| \notin T$. By our definition, $c'(v)= s + j - c(v)$. \\
	Therefore, $|c'(u)-c'(v)| = |s + j - c(u)-(s + j - c(v))|=|c(u)-c(v)|  \notin T $.\\
	Similarly, consider two vertices \textit{u} and \textit{v} such that $d(u,v)=2$. We know $c(u) \ne c(v)$.\\
	By our definition, let the complementary colours of $u$ and $v$ be  $c'(u)$ and $c'(v)$ respectively. Then $|c'(u)-c'(v)|=|s + j - c(u)-(s + j - c(v))|=|c(u)-c(v)|\ne 0$. Thus, the complementary colouring $c'$ also satisfies the condition for L(t, 1)-colouring. 
\end{proof}
Consider an L(t, 1)-colouring $c$. If $j=0$ for the complementary colouring $c'$ of the graph G, it is easy to note that the $c$-span and $c'$-span are identical. Moreover, if c is an optimal colouring of G so is $c'$ and both will give the same L(t, 1)-span for any graph G.\\
For $j\ne 0$, if the c-span of the colouring c is, say $l$, then the c-span for the complementary colouring $c'$ will be $l+j$.\\ 
The following remarks are also easily verifiable.
\begin{rem}
	L(t,1)-colouring is analogous to T-colouring when all the vertices of the graph are at a distance one i.e., for complete graphs $K_n$. Here L(t, 1)-span of G is same as T-span of G.
\end{rem}
\begin{rem}
	L(t, 1)-colouring is analogous to L(p, q)-colouring when q=1 and set T consists of all consecutive integers till (p-1). For such a set T, L(t, 1)-span of G is same as L(p, q)-span of G.
\end{rem}
	\begin{defn}
	Let the number of integers between 0 and max $\{T\}$ that are not available in T be $\sigma$. Hence $\sigma= max \{T\} - |T| + 1$.
\end{defn}
\subsection{L(t, 1)-colouring of star graph}
\par Star graphs are very important in the communication networks. We study the L(t, 1)-colouring of star graphs. At first, we see the effect on L(t, 1)-span of $K_{1,n}$ on changing the set T which will help us to find the upper bound for L(t, 1)-span of $K_{1,n}$. 
\begin{thm}
	The L(t, 1)-span of the star $K_{1,n}$ decreases with the increase in the value of $\sigma$ for $\sigma < n,$ if the maximum value of set T remains same. 
\end{thm}

\begin{proof}
	For a graph $K_{1,n}$, Let us consider a set T which consists of non-negative integers including 0, such that max \{T\}=r.\\
	Let $c_1 < c_2 <..< c_\sigma$ be the positive integers that are not in T.
	For $\sigma \ge n$, all the pendant vertices of the stars can be coloured from the set of missing colours i.e.,  \{ $c_1 < c_2 <...< c_\sigma$\}. The increase in the cardinality of $\sigma$ does not affect the colouring unless the new term missing from set T say $c_{\sigma +1}$ is less than $c_{\sigma}$.
	If $c_{\sigma +1}$ is less than $c_{\sigma}$, the pendant vertices can be coloured using colours $c_1 < c_2 <\ldots c_{\sigma-1}< c_{\sigma+1}$ instead of colours  $c_1 < c_2 <\ldots< c_\sigma$ which decreases the span of graph.\\
	If $c_{\sigma +1}$ is greater than $c_{\sigma}$, the colouring remains same, hence the span remains same.\\
	Suppose $\sigma < n$.
	Let $\lambda_{t,1} (K_{1,n})$=s.
Therefore, colours used to colour this graph will $\in$ \{0,1,...,s\}. In order to colour the star using L(t, 1)-colouring we start colouring the central vertex as 0 and then we use all missing colours less than r for pendant vertices, since all pendant vertices are at a distance 1 from the central vertex. For the remaining vertices we give colours $r+1, r+2, \ldots , s.$ Since the diameter of a star is 2 all pendant vertices should get different colours. Thus the above colouring is an L(t, 1)-colouring for a $K_{1,n}$ with span s.
	\par Consider a set T' with cardinality ($|T|-1$) and max\{T\}=r.\\
	Therefore, the number of the colours less than r, missing in set T' is given by say $\sigma'=\{\sigma + 1\}$.
	\par\textit{Case 1: $\sigma' < n$}\\
	Here, we can colour the graph $K_{1,n}$ in the same way as described above using colours  $c_1 < c_2 < \ldots < c_\sigma < c_{(\sigma + 1)}$ and continuing with colours r+1, r+2, ... and so on, since the number of vertices remains the same we already coloured one more vertex using missing colour from set T'. Hence we require colours till (s-1) for colouring this graph. Thus, we obtain an L(t, 1)-colouring for $K_{1,n}$ with span (s-1).
	\par\textit{Case 2: $\sigma' = n$}\\
	Here we can colour the graph $K_{1,n}$ in the same way as described above using colours  $c_1 <c_2<\ldots <c_\sigma<c_{(\sigma + 1)}$.Therefore all the missing colours saturate the vertices of $K_{1,n}$ 
	and $c_{(\sigma + 1)}$ becomes the span of the graph which is always less than r, as s is always greater than r.
\end{proof}
The following result is easily obtained by using above arguments.
\begin{cor}
	The L(t, 1)-span of the graph $K_{1,n}$ will remain same for all those sets T whose cardinality and maximum value are same i.e., it does not vary with the set T for $ n > \sigma.$
\end{cor}
\begin{thm}
	Let r be the $max\{T\}$. Then $\lambda_{t,1} (K_{1,n})
	= n-\sigma+r   ~\text{for }\sigma < n$
	and
	$\lambda_{t,1} (K_{1,n})< r ~\text{for }\sigma \ge n$.
\end{thm}
\begin{proof}
	\par\textit{Case 1: $\sigma < n$}\\
	From previous theorem, we know for $\sigma <n$, the graph $K_{1,n}$ can be coloured using colours $0$ for the central vertex and colours (say $c_1 < c_2 <\ldots< c_\sigma$.) less than r for pendant vertices. Since $\sigma <n$, we need to get the remaining $(n-\sigma)$ vertices to be coloured. Colouring these remaining vertices by using colours $r+1, r+2, \ldots$ and so on till $ r + (n-\sigma)$, we get $\lambda_{t,1} (K_{1,n})= r + (n-\sigma) $.
	\par$~~~~$\textit{Case 2: $\sigma \ge n$}\\
	For $\sigma \ge n$, all the pendant vertices can be coloured from the set of missing colours (say $c_1 < c_2 <\ldots< c_\sigma$) and since all these colours are less than r, we have $\lambda_{t,1} (K_{1,n}) < r$
\end{proof}
\subsection{L(t, 1)-colouring of  complete k-partite graph}
Here we find the bound for $\lambda_{t,1} (G)$ for a complete k-partite graph G.
\begin{thm}
	The L(t, 1)-span of complete $k$-partite graph for any finite set T, $	\lambda_{t,1} (K_{m_1,m_2,m_3,\ldots, m_k}) \le r(k) + \sum_{i=1}^{k}|m_i| -1 $,	where r is the max\{T\}.
\end{thm}
\begin{proof}
	Let G be a complete $k$-partite graph with partite sets $ m_1, m_2, m_3, \ldots, m_k$ with cardinalities $n_1, n_2, \ldots, n_k$. Choose any one partite set, select a vertex from it (for our convenience we take  partite set $m_1$ and vertex $v_1 \in m_1$ ). Colour $v_1$ using 0. Since all the other vertices of the other partite set apart from $m_1$ are at a distance 1 from $v_1$, none of them can have  their colour from the set T.\\
	Hence, to colour the vertices in the partite sets $ m_2, m_3, \ldots, m_k$, we start colouring the vertices of set $m_2$ using colour $r+1, r+2, \ldots, r+|m_2|$. We colour the vertices of set $m_3$ using the colours $r+|m_2|+r+1, r+|m_2|+ r+2, \ldots, r+|m_2|+r+|m_3|$, as each of vertices of $m_3$ is at a distance one from all the vertices of $m_1$.  Continuing in similar way, we can colour the vertices of the partite set $m_k$ by using colours $r+|m_2|+r+|m_3|+\ldots+r+|m_{k-1}|+r+1, r+|m_2|+r+|m_3|+\ldots+r+|m_{k-1}|+r+2, \ldots, r+|m_2|+r+|m_3|+\ldots+r+|m_{k-1}|+r+|m_k|$.\\
	Since vertices $v_2, v_3, \ldots, v_{|m_1|}$ in the partite set $m_1$ need to be coloured, which are at distance two from $v_1$ and at a distance one from all the vertices that are already coloured we use the colour $r+|m_2|+r+|m_3|+\ldots+r+|m_{k-1}|+r+|m_k|+ r+1$ for $v_2$. The remaining vertices can be coloured by increasing colour of $v_2$ by one for each vertex. Hence the colours required are $r+|m_2|+r+|m_3|+\ldots+r+|m_{k-1}|+r+|m_k|+ r+1,\ldots, r+|m_2|+r+|m_3|+\ldots+r+|m_{k-1}|+r+|m_k|+ r+|m_1|-1$ $=rk +\sum_{i=1}^{k}|m_i| -1$.
	Therefore, the L(t, 1)-span of the complete k-partite graph becomes $rk +\sum_{i=1}^{k}|m_i| -1$. We achieve this bound only when T has the elements \{0, 1,$\ldots$, r\}. In all other cases the $\lambda{t,1}(K_{m_1,m_2,m_3,...m_k})$ is less than $rk +\sum_{i=1}^{k}|m_i| -1$. Hence,$\lambda_{t,1} (K_{m_1,m_2,m_3,\ldots,m_k}) \le r(k) + \sum_{i=1}^{k}|m_i| -1 $.
\end{proof} 

\section{Conclusion}
\par In this paper, we have introduced L(t, 1)-colouring of graphs. This is a new addition to family of channel assignment problems. The significance of the set T is that certain frequencies can be reserved or classified for the requirements of the state. L(t, 1)-colouring also assures that two radio channels to have the same frequency if they are at a distance of at least three. In this way the disturbances due to frequency interference can be minimized.


%
%




\end{document}